\documentclass{amsart}
\usepackage{graphicx}
\usepackage{amssymb}
\newtheorem{thm}{Theorem}[section]
\newtheorem{cor}[thm]{Corollary}
\newtheorem{lem}[thm]{Lemma}
\newtheorem{prop}[thm]{Proposition}
\theoremstyle{definition}
\newtheorem{defn}[thm]{Definition}
\newtheorem{rem}[thm]{Remark}
\theoremstyle{question}

\theoremstyle{Conjecture}
\newtheorem{con}[thm]{Conjecture}
\numberwithin{equation}{section}

\begin{document}

\title[Nonnilpotent subsets in the Suzuki groups]{Nonnilpotent subsets in the Suzuki groups}%
\author{Mohammad Zarrin}%

\address{Department of Mathematics, University of Kurdistan,
Sanandag
 66177-15175, Iran}%
 \email{M.Zarrin@uok.ac.ir, Zarrin@ipm.ir}
\begin{abstract}
Let $G$ be a group and $\mathcal{N}$ be the class of nilpotent
groups. A subset $A$ of $G$ is said to be nonnilpotent if for any
two distinct elements $a$ and $b$ in $A$, $\langle a, b\rangle
\not\in \mathcal{N}$. If, for any other nonnilpotent subset $B$ in
$G$, $|A|\geq |B|$, then $A$ is said to be a maximal nonnilpotent
subset and the cardinality of this subset (if it exists) is
denoted by $\omega(\mathcal{N}_G)$. In this paper, among other
results, we obtain $\omega(\mathcal{N}_{Suz(q)})$ and
$\omega(\mathcal{N}_{PGL(2,q)})$, where $Suz(q)$ is the Suzuki
simple group over the field with $q$ elements and $PGL(2,q)$ is
the projective general linear group of degree $2$ over the finite
field of size $q$, respectively. 
 \end{abstract}
\maketitle

\section{\textbf{Introduction and results}}

   One can associate a graph to a group in many different ways
 (see for example \cite{Ab1, Ab3, z, Bal}). Let $G$ be a
group. Following \cite{z}, we shall use the notation
$\mathcal{N}_G$ to denote the
 nonnilpotent graph, as follows:
 take $G$ as the vertex set  and   two  vertices are adjacent
  if they  generate a nonnilpotent subgroup. Note that if $G$ is weakly nilpotent (i.e., every
two generated subgroup of $G$ is nilpotent), $\mathcal{N}_G$ has
no edge. It follows that the nonnilpotent graphs of weakly
nilpotent groups with the same cardinality are isomorphic. So we
must be interested in non weakly nilpotent groups. A set $C$ of
vertices of a graph $\Lambda$ whose induced subgraph is a complete
subgraph is called a clique and the maximum size (if it exists) of
a clique in a graph is called the clique number of the graph and
it is denoted by $\omega(\Lambda)$. A subset $A$ of $G$ is said to
be nonnilpotent if for any two distinct elements $a$ and $b$ in
$A$, $\langle a, b\rangle \not\in \mathcal{N}$ (we call two
elements $a, b$ nonnilpotent). If, for any other nonnilpotent
subset $B$ in $G$, $|A|\geq |B|$, then $A$ is said to be a maximal
nonnilpotent subset. Thus $\omega(\mathcal{N}_G)$ is simply the
cardinality of maximal nonnilpotent subset (or the maximum number
of pairwise nonnilpotent elements) in the group $G$. One of our
motivations for associating with a group such kind of graph is a
problem posed by Erd\"{o}s: For a group $G$, consider a graph
$\mathcal{A}_G$ whose vertex set is $G$ and join two distinct
elements if they do not commute. Then he asked: Is there a finite
bound for the cardinalities of cliques in $\mathcal{A}_G$, if
$\mathcal{A}_G$ has no infinite clique?\; Neumann \cite{Ne}
answered positively Erd\"{o}s' problem by proving that such groups
are exactly the center-by-finite groups and the index of the
center can be considered as the requested bound in the problem.

This results suggests that the clique number of a graph of a group
not only has some influence on the structure of a group but also
finding it, it is important and interesting. Recently Abdollahi
and Zarrin in \cite{z}, have studied the influence of
$\omega(\mathcal{N}_G)$ on the structure of a group. Then Azad in
\cite{Az}, obtained a lower bound for
$\omega(\mathcal{N}_{GL(n,q)})$ and he determined
$\omega(\mathcal{N}_{PSL(2,q)})$. Also the author in \cite{zar},
shortly prove, determined $\omega(\mathcal{N}_{PSL(2,q)})$ (see
 Proposition 4.2 of \cite{zar}). Clearly, $\mathcal{N}_G$ is a
subgraph of $\mathcal{A}_{G}$ and so
$$\omega(\mathcal{N}_{G})\leq \omega(\mathcal{A}_{G}).\eqno(\star)$$

A group $G$ is an $AC$-group if $C_G(g)$ is abelian
  for  all $g\in G\setminus Z(G)$, where $C_G(g)$
is the centralizer of the element $g$ in $G$.
\begin{rem}
 Let $G$ be a centerless $AC$-group. Then it follows from
Lemma 4.1 of \cite{zar}, that
$\omega(\mathcal{N}_{G})=\omega(\mathcal{A}_{G})$ (note that, as
$G$ is an $\mathrm{AC}$-group, we have either
 $C_G(a)\cap C_G(b)=Z(G)$ or $ab=ba$ for every $a, b\in G$).
\end{rem}

 In this paper we give some properties of $\mathcal{N}(n)$-groups, where by an $\mathcal{N}(n)$-group
  we mean a group $G$ which has exactly  $n~nilpotentizer$ (in fact, by using this class of groups we give
an upper bound for $\omega(\mathcal{N}_{G})$ in terms of $n$).
Also we determine $\omega(\mathcal{N}_{Suz(q)})$ and
$\omega(\mathcal{N}_{PGL(2,q)})$. Finally, by using these results,
we characterize all nonabelian finite semisimple groups $G$ with
$\omega(\mathcal{N}_{G})\leq 72$, in fact we will
generalize Theorem 4.5 of~\cite{Az}.\\

Our main results are:

\begin{thm}\label{t1}
Let $G = Suz(q)$ \rm($q = 2^{2m+1}$ and $m > 0$). Then
$$\omega{(\mathcal{N}_G)}= (q^2 + 1)+ \frac{q^2(q^2 + 1)}{2} +
\frac{q^2(q^2 + 1)(q - 1)}{4(q + 2r + 1)}+\frac{q^2(q^2 + 1)(q
-1)}{4(q-2r+1)},$$ where $r=\sqrt{\frac{q}{2}}$.
\end{thm}
\begin{thm}\label{t2}
We have
\begin{enumerate}
\item If $q=4$ or $q>5$, then
$\omega(\mathcal{N}_{PGL(2,q)})=\omega(\mathcal{N}_{GL(2,q)})=
    q^2+q+1.$
\item  $1015\leq
\omega(\mathcal{N}_{PSL(3,3)})=\omega(\mathcal{N}_{SL(3,3)}).$
\end{enumerate}
\end{thm}
In particular, we conjecture that
$\omega(\mathcal{N}_{PSL(3,3)})=\omega(\mathcal{N}_{PGL(3,3)})=1015.$
\begin{thm}\label{t3}
Let $G$ be a nonabelian finite semisimple group. Then
$\omega(\mathcal{N}_{G})\leq 72$ if and only if $G\cong A_5,
PSL(2,7), S_5$ or $PGL(2,7)$.
\end{thm}
From Theorem \ref{t3} we obtain a nice characterization for $A_5$
and $PSL(2,7)$ (see Corollary  \ref{co1}, below).\\

 Throughout this paper all groups are finite and we will
use the usual notation, for example $A_n,~ S_n, ~
SL(n,q),~GL(n,q), ~ PSL(n,q),~PGL(n,q)$ and $Suz(q)$
\rm($q=2^{2m+1}$ and $m>0$), respectively, denote the alternating
groups on $n$ letters, symmetric group on $n$ letters, special
linear group of degree $n$ over the finite field of size $q$,
general linear group of degree $n$ over the finite field of size
$q$, projective special linear group of degree $n$ over the finite
field of size $q$, projective general linear group of degree $n$
over the finite field of size $q$ and the Suzuki group over the
field with $q$ elements.

\section{\textbf{ $\mathcal{N}(n)$-groups}}

Let $G$ be a group. Recall that the centralizer of an element
$a\in G$ can be defined by $C_G(a)=\{ b\in G\mid \langle a,
b\rangle ~\mbox{is~ abelian} \}$ and it is a subgroup of $G$. If,
in the above definition, we replace the word "abelian" with the
word "nilpotent" we get a subset of $G$, called the {\sl
nilpotentizer} of an element $a\in G$ (see also \cite{z}). In
fact, this subset is an extension of the centralizer. To be
explicit, define the {\sl nilpotentizer} of an element $a\in G$,
denoted by $nil_G(a)$, by
\[ nil_G(a)=\{b\in G\mid\langle a, b\rangle ~~\textrm{is
nilpotent}\}.\]

Also for a nonempty subset $S$ of $G$, we define the nilpotentizer
of $S$ in $G$, to be
$$nil_G(S)=\bigcap_{x\in S} nil_G(x).$$  When $S=G$; we call $nil_G(G)$ the
nilpotentizer of $G$, and it will be  denoted  by $nil(G)$. Thus
$$nil(G)=\{x\in G \mid\langle x, y \rangle ~~\textrm{is nilpotent for all}\; y\in G\}.$$
It is not known whether the subset $nil(G)$ is a subgroup of $G$,
but in many important cases it is a subgroup. In particular,
$nil(G)$ is equal to the hypercenter $Z^*(G)$ of $G$ whenever $G$
satisfies the maximal condition on its subgroups or $G$ is a
finitely generated solvable group (see Proposition 2.1 of
\cite{z}). Also note that in general for an element $x$ of a group
$G$, $nil_G(x)$ is not a subgroup of $G$. For example, in the
group $G=S_4$, clearly the element $u=(13)(24)$ belongs to
$O_2(S_4)$, so $nil_G(u)$ contains the union of all Sylow
$2$-subgroups. The only other elements of $S_4$ are $3$-cycles.
Since none of these centralize $u$, $nil_G(u)$ is exactly the
union of all three Sylow $2$-subgroups, that is $nil_G(u)$ is not
subgroup \rm(see also Lemma 3.3 of
 \cite{z}).

  \begin{defn}
 We say that a group $G$ has $n~nilpotentizer$
(or that $G$ is an $\mathcal{N}(n)$-group) if $|nilp(G)|= n$,
where $nilp(G) = \big\{nil_G(g)\mid g\in G\big\}$.
\end{defn}
 It is clear that
a group is an $\mathcal{N}(1)$-group if and only if it is weakly
nilpotent. One of our motivations for the above definition is the
following Proposition. (In fact,  for a nonweakly nilpotent group
$G$, we give some interesting relations between
$\omega(\mathcal{N}_{G})$ and $|nilp(G)|$.)
\begin{prop}\label{l00}
Let $G$ be a non weakly nilpotent group. Then we have\\
(1)\; $\omega(\mathcal{N}_{G}) + 1 \leq
|nilp(G)|$.\\
(2)\; If every nilpotentizer of $G$ is a nilpotent subgroup
{\rm(}or $G$ is an $\mathcal{N}n$-group, see \cite{Az}{\rm)}, then
$\omega(\mathcal{N}_{G}) + 1 = |nilp(G)|$.
\end{prop}
\begin{proof}
(1) Assume that $X=\{a_1, \ldots, a_t\}$ be an arbitrary clique
for the graph $\mathcal{N}_G$. It follows that $nil_G(a_i)\not
=nil_G(a_j)$ for every $1\leq i<j\leq t$. From which it follows
that $|X|+1 \leq |nilp(G)|$, as $nil_G(e)=G$ where $e$ is the
trivial element of $G$, and so $\omega(\mathcal{N}_{G}) + 1 \leq
|nilp(G)|$.

 (2)\; Let $A= \{a_1,\ldots, a_n\}$
be a maximal nonnilpotent subset of $G$ and $nil_G(x)$ be a proper
nilpotentizer of $G$. Then there exists a $1\leq i\leq n$ such
that $x\in nil_G(a_i)$ by the maximality of $A$. Since
$nil_G(a_i)$ is a nilpotent subgroup, $nil_G(a_i)\leq nil_G(x)$.
On the other hand $a_i\in nil_G(x)$ gives $nil_G(x)\leq
nil_G(a_i)$. Therefore $nil_G(x)= nil_G(a_i)$ and so
$|nilp(G)|=|\{G,nil_G(a_i) \mid 1\leq i \leq n \}|$. This
completes the proof.
\end{proof}
Here, we give some properties about $\mathcal{N}(n)$-groups.
\begin{prop}\label{l100}
Let $G$ be an $\mathcal{N}(n)$-group. Then\\
(1)\; $G$ is nilpotent if and only if  $n<5$. The group $S_3$ is
an
$\mathcal{N}(5)$-group;\\
(2) If $n<22$, then $G$ is solvable. The group $A_{5}$ is an
$\mathcal{N}(22)$-group.
\end{prop}
\begin{proof}
The result follows from part (1) of Proposition \ref{l00} and the
main result of \cite{En}.
\end{proof}
\begin{cor}
There is no finite $\mathcal{N}(n)$-group with $n\in\{2, 3, 4\}$.
\end{cor}
Here (for infinite groups) we prove that every arbitrary
$\mathcal{N}(n)$-group with $n=4$ is an Engel group. Recall that a
group $G$ is an Engel group if for each ordered pair $(x, y)$ of
elements in $G$ there is a positive integer $n=n(x, y)$ such that
$[x,_ny]=1$, where $[x, y] = x^{-1}y^{-1}xy=x^{-1}x^y$ and
$[x,_{m+1} y] = [[x,_{m} y], y]$ for all positive integers $m$.
\begin{thm}\label{ts2}
Let $G$ be an \rm{(}not necessarily finite\rm{)}
$\mathcal{N}(n)$-group with $n=4$. Then it is an Engel group.
\end{thm}
\begin{proof} Suppose that $G$ is an $\mathcal{N}(n)$-group with $n=4$. We claim
that for  every two arbitrary
$x, y\in G$, the subgroup $\langle x, x^y\rangle$ is nilpotent.\\
To see that, suppose, a contrary, that  there exist  $x, y\in G$
such that $\langle x, x^y\rangle$ is not nilpotent. Now we
consider the set
$$N=\{nil_G(x), nil_G(y), nil_G(xy), nil_G(x^y)\}.$$
  As $|nilp(G)|=4$, it follows that there exist at least two
elements $nil_G(a),~nil_G(b)\in N$ such that $nil_G(a)=nil_G(b)$.
It implies that $\langle x, y\rangle$ or $\langle x, x^y\rangle$
is a nilpotent group. Now since $\langle x, x^y\rangle \leq
\langle x, y\rangle$, it follows that $\langle x, x^y\rangle$ is
nilpotent, a contrary. Hence for every two arbitrary $x, y\in G$,
$\langle x, x^y\rangle$ is nilpotent and so $\langle x,
{x^{-1}}^y\rangle$ is nilpotent. It follows that
$[y^{-1}x^{-1}y,_tx]=1$ for some $t\in\mathbb{N}$. Now by using
the relation $[ab, c]=[a,c]^b[b, c]$ we get $[y,_{t+1}x]=1$,
namely $G$ is an Engel group, as required.
\end{proof}

\begin{lem}\label{l0}
Suppose that $G$ is a group and $H\leq G$. Then $|nilp(H)|\leq
|nilp(G)|$.
\end{lem}
\begin{proof}
It follows easily from
$$nil_H(h)=nil_G(h)\cap H$$ for every element $h\in H$.
\end{proof}
\begin{lem}\label{lj1}
Suppose that $G_i$ is a finite group \rm{(}$i=1,\dots, t$\rm{)}.
Then $$|nilp(G_1\times G_2\times\dots\times
G_m)|=\prod_{i=1}^m|nilp(G_i)|.$$
\end{lem}
\begin{proof}
Put $H=\prod_{i=1}^m G_i$. Applying Proposition 3.1 of \cite{zar},
we can show that $nil_{H}(x_1,
\dots,x_t)=nil_{G_1}(x_1)\times\dots\times nil_{G_t}(x_t)$, for
all $(x_1,\dots,x_t)\in H$. It follows that for every $1\leq i\leq
t$ we have $$nil_{H}(x_1,\dots,x_t)=nil_H(y_1,y_2,\dots,y_t)$$ if
and only if $nil_{G_i}(x_i)=nil_{G_i}(y_i)$ and the result
follows.
\end{proof}

\begin{lem}\label{lt}
 Let $G$ be a group. Then
  $|nilp(G)|\geq |nilp(\frac{G}{Z^*(G)})|$.
\end{lem}
\begin{proof}
It is clear that for all $x\in G$,
$nil_{\frac{G}{Z^*(G)}}(xZ^*(G))=\frac{nil_G(x)}{Z^*(G)}$. It
follows that if $nil_{\frac{G}{Z^*(G)}}(xZ^*(G))\neq
nil_{\frac{G}{Z^*(G)}}(yZ^*(G))$, then $nil_G(x)\neq nil_G(y)$.
 This completes this proof.
\end{proof}

\section{\textbf{Proofs of Theorem \ref{t1} and Theorem \ref{t2}}}

To prove Theorem \ref{t1} we need the following lemmas.
\begin{lem}{\rm\cite[Lemma 3.7]{z}}\label{lr}
Let $G$ be a group and $H$  a nilpotent subgroup of $G$ in which
$C_G(x)\leq H$ for every $x\in H\backslash\{1\}$. Then
$nil_G(x)=H$ for every $x\in H\backslash\{1\}$.
\end{lem}
\begin{prop}{\rm\cite[Proposition 2.1]{z}}\label{ln0}
Let  $G$ be  any group. Then
\begin{enumerate}
\item $Z^{*}(G)\subseteq nil(G) \subseteq R(G)$, where $R(G)$ is
the set of right Engel elements of $G$. \item If $G$ satisfies the
maximal condition on its subgroups or $G$ is finitely generated
solvable group then $Z^{*}(G)=nil(G)=R(G)$.
\end{enumerate}
\end{prop}
\begin{lem}\label{lz}
Let $G$ be a group, $H\leq G$ and let $M_1, \cdots, M_n$ be non
trivial proper subgroups of $H$ such that $H=\bigcup_{i=1}^nM_i$
and $M_i\bigcap M_j=nil(H)$ for $i\neq j$. If $nil_H(g)\subseteq
M_i$ for all $g\in M_i\setminus nil(H)$, then
$$\omega(\mathcal{N}_{H})=\sum_{i=1}^n\omega(\mathcal{N}_{M_i}).$$
\end{lem}
\begin{proof}
If $N$ is any clique of $\mathcal{N}_{H}$ then $N=\bigcup_{i=1}^n
N_i$ where $N_i\subseteq M_i\setminus nil(H)$ for each
$i\in\{1,\cdots n\}$. It follows that $|N|\leq
\sum_{i=1}^n\omega(\mathcal{N}_{M_i})$. Now let $W_i$ be a maximum
clique of $\mathcal{N}_{M_i}$ for each $i\in\{1,\cdots n\}$. We
claim that $W=\bigcup_{i=1}^nW_i$ is a maximum clique for
$\mathcal{N}_{H}$. Suppose for a contradiction that there exist
two distinct elements $a$ and $b$ in $\bigcup_{i=1}^nW_i$ such
that $\langle a, b\rangle$ is a nilpotent group. Thus there exist
$i\neq j$ such that $a\in M_i$ and $b\in M_j$. Therefore $a\in
nil_H(b)\subseteq M_j$. It follows that $H=nil_H(a)\subseteq M_i$
and so $H=M_i$, a contradiction. Now as
$|W|=\sum_{i=1}^n\omega(\mathcal{N}_{M_i})$, the result follows.
\end{proof}
A set $\mathcal{P}=\{H_1,\ldots, H_n\}$ of subgroups $H_i
~~(i=1,\ldots, n)$ is said to be a partition of $G$ if every
non-identity element $x\in G$ belongs to one and only one subgroup
$H_i\in \mathcal{P}$.\\

 \noindent{\bf Proof of
Theorem \ref{t1}.} The Suzuki group $G$ contains subgroups $F, A,
B$ and $C$ such that $|F|=q^2, |A|=q-1, |B|=q-2r+1$ and
$|C|=q+2r+1$ (see \cite[Chapter XI, Theorems 3.10 and 3.11]{Hu2}).
Also by \cite[pp. 192-193, Theorems 3.10 and 3.11]{Hu2}, the
conjugates of $A, B, C$ and $F$ in $G$ form a partition, say
$\mathcal{P}$,  for $G$, and $A, B, C$ are cyclic and $F$ is a
Sylow $2$-subgroup of $G$ and also for every  $M\in \mathcal{P}$
we have $C_{G}(b)\leq M$ for all $b\in M$.

 Assume that $a\in G\setminus \{1\}$. Since $\mathcal{P}$ is a partition of $G$,  $a\in M$ for
  some $M\in \mathcal{P}$. It follows from Lemma \ref{lr}
  that  $nil_{G}(a)=M$ for all $a\in M$ \rm($\star\star$).
From \cite[Chapter XI, Theorems 3.10 and 3.11]{Hu2} implies that
the number of conjugates of $C$, $B$, $A$ and $F$ in $G$ are
respectively, $k= \frac{q^2(q-1)(q^2+1)}{4(q+2r+1)}$ ,
$n=\frac{q^2(q-1)(q^2+1)}{4(q-2r+1)}$, $s=\frac{q^2(q^2+1)}{2}$
and $t=q^2+1$. Since $G$ is a finite simple group, Proposition
\ref{ln0} follows that $nil(G)=Z^{*}(G)=1$. Now as $\mathcal{P}$
is a partition for $G$ and by ($\star\star$), we implies, by Lemma
\ref{lz}, that $\omega(\mathcal{N}_{G})$ is equal to size of the
set $\mathcal{P}$ (note that, for a non trivial nilpotent group
$H$, we define $\omega(\mathcal{N}_H)=1$). Thus
$\omega(\mathcal{N}_G)=k+n+s+t$.
This completes the proof.\\

Now in view of the proof of Theorem \ref{t1}, one can see that
every nilpotentizer of $Suz(q)$ is a nilpotent subgroup (and hence
$Suz(q)$ is an $\mathcal{N}n$-group, see \cite{Az}). Therefore, by theorem
\ref{t1} and by Part (2) of Proposition \ref{l00}, we give the
following interesting result.

\begin{cor}
Let $G = Suz(q)$ \rm($q = 2^{2m+1}$ and $m > 0$). Then
$$|nilp(G)|=(q^2 + 2)+ \frac{q^2(q^2 + 1)}{2} +
\frac{q^2(q^2 + 1)(q - 1)}{4(q + 2r + 1)}+\frac{q^2(q^2 + 1)(q
-1)}{4(q-2r+1)},$$ where $r=\sqrt{\frac{q}{2}}$.
\end{cor}

To prove Theorem \ref{t2} we need the following Lemma.
\begin{lem}\label{lm1}
Let $G$ be a finite group. Then: \\
(1)\; For any subgroup $H$ of $G$, $\omega(\mathcal{N}_{H})\leq
\omega(\mathcal{N}_{G})$;\\
(2)\; For any nonabelian factor group $\frac{G}{M}$ of $G$,
$\omega(\mathcal{N}_{\frac{G}{M}})\leq
\omega(\mathcal{N}_{G})$;\\
(3)\; $\omega(\mathcal{N}_{\frac{G}{K}})=\omega(\mathcal{N}_{G})$,
where $K$ is a normal subgroup of $G$ with $K\leq Z^*(G)$.
\end{lem}
\begin{proof}
(1)\;Clearly.\\
(2)\; This is straightforward.\\
(3)\; For proof it is enough to note that if $\frac{\langle a,
b\rangle}{Z^*(G)\cap \langle a, b\rangle }$ is nilpotent, then
$\langle a, b\rangle$ is nilpotent, for all $a, b\in G$.
\end{proof}
\begin{cor}
$$\omega(\mathcal{N}_{PGL(n,q)})=
 \omega(\mathcal{N}_{GL(n,q)}),$$ and $$\omega(\mathcal{N}_{PSL(n,q)})=
 \omega(\mathcal{N}_{SL(n,q)}).$$
\end{cor}

\noindent{\bf Proof of Theorem \ref{t2}.} (1)\; Since
$\omega(\mathcal{N}_{PGL(n,q)})\leq
 \omega(\mathcal{A}_{PGL(n,q)})$ and $PSL(n,q)\cong \frac{ZSL(n,q)}{Z}\leq PGL(n,q)$ where $Z$ is the center of $GL(n,q)$,
 it follows, by part (1) of Lemma \ref{lm1}, that $$\omega(\mathcal{N}_{PSL(2,q)})\leq \omega(\mathcal{N}_{PGL(2,q)})=
 \omega(\mathcal{N}_{GL(2,q)})\leq
 \omega(\mathcal{A}_{PGL(2,q)}\leq
 \omega(\mathcal{A}_{GL(2,q)}).$$ Now the result follows from  Proposition 4.3 of \cite{Ab3} and Proposition 4.2 of
 \cite{zar}.

 (2)\; We have used the following function written with {\sf GAP} \cite{Ga} program to
prove this part of the theorem. The input of the function is a
group $G$ and an element $t\in G$ and the output is $nil_G(t)$. \\\\
{\rm f:= function(G,t)~~~~ local~ r;}\\
{\rm r:=Set(Filtered(G,i$\rightarrow$
IsNilpotent(Group(i,t))=true));}\\{\rm return r; end;}

Let $G=PSL(3,3)$. Therefore it is easy to see that the set of
order elements of $nilp(G)$ is
$$\{6, 13, 16, 27, 32, 162, 192, 5616\}$$ and if $A = \{nil_G(g)
\mid g\in G, |nil_G(g)| = 6\}$, $B = \{nil_G(g) \mid g\in G,
|nil_G(g)| = 16\}$, $C = \{nil_G(g) \mid g\in G, |nil_G(g)| =
13\}$ and $D = \{nil_G(g) \mid g\in G, |nil_G(g)|=27\}$, then $|A|
= 468, ~|B| = 351, |C| = 144$ and $|D| = 52$.  Thus there exist
elements $a_i, b_j, c_k, d_l\in G$ such that $|nil_G(a_i)|=6$ for
$1\leq i\leq 468$, $~~|nil_G(b_j)|=16$ for $1\leq j\leq 351$,
$|nil_G(c_k)|=13$ for $1\leq k\leq 144$ and $|nil_G(d_l)|=27$ for
$1\leq l\leq 52$. Set $X=\{a_1,\ldots, a_{468}, b_1,\ldots,
b_{351}, c_1,\ldots,c_{144}, d_1,\ldots, d_{52}\}$. Also we can
show that, by {\sf GAP} \cite{Ga}, that each element in $A\cup B
\cup C\cup D$ is a nilpotent subgroup of $G$.

Now we claim that $X$ is a nonnilpotent subset of $G$. Let $a,
b\in X$  such that $\langle a,b\rangle$ is nilpotent. Therefore
$a\in nil_G(b)$. Since $nil_G(b)$ is a nilpotent subgroup, it
follows that $nil_G(b)\leq nil_G(a)$. Similarly, $nil_G(a)\leq
nil_G(b)$. Hence $nil_G(a)=nil_G(b)$, which is a contradiction.
Thus $X$ is a nonnilpotent subset of $G$ and so $1015=|X|\leq
\omega(\mathcal{N}_{G})$. This completes the proof.\\

Note that by Theorem 4.1 of \cite{Az}, we get $52\leq
\omega(\mathcal{N}_{GL(3,3)})$. But by Theorem \ref{t2} and as
$\omega(\mathcal{N}_{PSL(n,q)})\leq
\omega(\mathcal{N}_{PGL(n,q)})$, we can obtain that $1015\leq
\omega(\mathcal{N}_{GL(3,3)})=\omega(\mathcal{N}_{PGL(3,3)})$.
Finally, in view of these results, we state the following
conjecture:
\begin{con}\label{con1}
$\omega(\mathcal{N}_{PSL(3,3)})=\omega(\mathcal{N}_{PGL(3,3)})=1015.$
\end{con}
\section{\textbf{Proof of Theorem \ref{t3}}}

To prove Theorem \ref{t3} we need the following lemma.
\begin{lem}\label{l11}
 If $G$ is a nonabelian finite simple group with $\omega(\mathcal{N}_{G})\leq72$.
Then $G$ is isomorphic to one of the following groups:\\
 $(1)$ $G\cong PSL(2,5)$;\\
 $(2)$ $G\cong PSL(2,7)$.
\end{lem}
\begin{proof}
 The proof follows easily from Corollary 3.4 of \cite{Az} and by an argument similar to the proof of Theorem
4.5 of \cite{Az} (note that, by Theorem \ref{t1}, for the suzuki
group $Suz(2^{2m+1})$ we have
$\omega(\mathcal{N}_{Suz(2^{2m+1})})> 73$  and also it is easy to
see (e.g., by  {\sf GAP} \cite{Ga}), that the number of Sylow
$43$-subgroups of the projective special unitary group of degree
three over the finite fields of order 7, $PSU(3,7)$ is more than
$73$).
\end{proof}

 Recall that a group $G$ is semisimple if $G$ has no non-trivial
normal abelian subgroups. If $G$ is a finite group then we call
the product of all minimal normal nonabelian subgroups of $G$ the
centerless $CR$-radical of $G$; it is a direct product of
nonabelian simple groups \cite[see page 88]{d.j.r}.\\
\begin{lem}\label{l1}
Suppose that $G_i$ is a finite group \rm{(}$i=1,\dots, t$\rm{)}.
Then $$\omega{(\mathcal{N}_{G_1\times G_2\times\dots\times
G_m})}=\prod_{i=1}^m \omega{(\mathcal{N}_{G_i})}.$$
\end{lem}
\begin{proof}
This follows from Proposition 3.1 of \cite{zar}.
\end{proof}

\noindent{\bf Proof of Theorem \ref{t3}.}
 Assume that  $R$ is the centerless $CR$-Radical of $G$. Then $R$
is a direct product of a finite number, $t$, of finite nonabelian
simple groups, say $R= S_1\times\cdots\times S_t$. Since
$\omega(\mathcal{N}_{G})\leq 72$ and $S_i\leq G$,  it follows that
$\omega(\mathcal{N}_{S_i})\leq 72$ ($\omega{(\mathcal{N}_{R})}\leq
72$). Therefore Lemma \ref{l11}, follows that  $S_i\cong
PSL(2,5)~\mbox{or}~~PSL(2,7)$ for $i\in \{1,\cdots,t\}$.

On the other hand by Lemma \ref{l1}, and also since
$\omega{(\Gamma_R)}\leq 72$, we have $t=1$. Therefore $R \cong
PSL(2,5)~\mbox{or}~~PSL(2,7)$. But we know that $C_G(R) = 1$ and
$G$ is embedded into $Aut(R)$. Hence
 $G\cong PSL(2,5),~~PSL(2,7), S_5$  or $PGL(2,7)$. But for these
 groups we have $\omega(\mathcal{N}_{G})\leq \omega(\mathcal{A}_{G})\leq 57\leq
 72$ and the result follows.

By Theorem \ref{t3}, we have a nice characterization for the
following simple groups.
\begin{cor}\label{co1}
Let $G$ be a finite simple group. Then we have
\begin{quote}
$G\cong A_5$ if and only if $\omega(\mathcal{N}_{G})=21$;\\
$G\cong PSL(2,7)$ if and only if $\omega(\mathcal{N}_{G})=57$.
\end{quote}
\end{cor}

\end{document}